\documentclass[11pt]{amsart}
\usepackage{amssymb, latexsym, mathrsfs}
\usepackage{tikz-cd} 

\usepackage[colorlinks=true, pdfstartview=FitV,linkcolor=blue,citecolor=blue,urlcolor=blue]{hyperref}

\numberwithin{equation}{section}
\newtheorem{Thm}[equation]{Theorem}
\newtheorem{Prop}[equation]{Proposition}
\newtheorem{Cor}[equation]{Corollary}
\newtheorem{Lem}[equation]{Lemma}

\theoremstyle{definition}
\newtheorem{Def}[equation]{Definition}

\newtheorem{Rmk}[equation]{Remark}

\newcommand{\boi}{\boldsymbol{\iota}}
\newcommand{\bi}{\mathbf{i}}
\newcommand{\bj}{\mathbf{j}}

\begin{document}

\title [Representations of spin quiver Hecke algebras]
{Representations of spin quiver Hecke algebras \\for orthosymplectic Lie superalgebras}
\author[K. Christodoulopoulou]{Konstantina Christodoulopoulou}
\address{Department of
Mathematics, University of Florida, Gainesville, FL  32611-8105, U.S.A. }
\email{kchristod@ufl.edu}
\author[K.-H. Lee]{Kyu-Hwan Lee$^{\star}$}
\thanks{$^{\star}$This work was partially supported by a grant from the Simons Foundation (\#318706).}
\address{Department of
Mathematics, University of Connecticut, Storrs, CT 06269-3009, U.S.A.}
\email{khlee@math.uconn.edu}
\subjclass[2010]{Primary 17B37}
\begin{abstract}
We construct  all the irreducible representations of spin quiver Hecke algebras for  orthosymplectic Lie superalgebras $osp(1|2n),$ and show that their highest weights are given by the dominant words. We use the dominant Lyndon words to construct the cuspidal modules and  show that the irreducible representations are the simple  heads of standard representations constructed by induction from the cuspidal modules. 
\end{abstract}

\maketitle

\section*{Introduction}

Introduced by Khovanov and Lauda \cite{Khov2009} and independently by Rouquier \cite{Rouq2008}, the Khovanov-Lauda-Rouquier (KLR) algebras (also known as quiver Hecke algebras) have attracted much attention as these algebras categorify the lower (or upper) half of a quantum group.  More precisely, the Cartan datum associated with a Kac-Moody algebra $\mathfrak g$ gives rise to a KLR algebra $R$.  The category of finitely generated projective graded  modules of this algebra can be given a bialgebra structure by taking the Grothendieck group, and taking the induction and restriction functors as multiplication and co-multiplication. It turns out that this bialgebra is isomorphic to Lusztig's integral form of $U_q^{-}(\mathfrak g)$, and in this sense we say that the KLR algebra $R$ categorifies the negative part  $U_q^{-}(\mathfrak g)$ of the quantum group.

In the study of the category of representations, it is of fundamental interest to construct irreducible representations of $R$. In the paper \cite{KR}, Kleshchev and Ram defined a class of {\em cuspidal} representations for finite types, showed that every irreducible representation appears as the head of some induction of these cuspidal modules, and constructed almost all cuspidal representations. Hill, Melvin, and Mondragon in \cite{Hill2012} completed the construction of cuspidal representations in all finite types, and re-framed them in a more unified manner. Using a different approach, Benkart, Kang, Oh, and Park in \cite{Benk2011} also constructed irreducible representations utilizing a crystal structure on the isomorphism classes of irreducible representations of a KLR algebra obtained by 
Lauda and Vazirani in \cite{Lauda2011}. 

Along these developments, the case of  Kac--Moody superalgebras has been considered. 
As a foundational work in the superalgebra case, Kang, Kashiwara, and Tsuchioka generalized the KLR algebras to the spin quiver Hecke algebras \cite{KKT}. Subsequently, Hill and Wang \cite{HW} and Kang, Kashiwara, and Oh \cite{KKO13, KKO14} showed that the spin quiver Hecke algebras provide a categorification of  half of  quantum  Kac--Moody superalgebras. It is well known that a Kac-Moody superalgebra can be associated to a generalized Cartan matrix. The only finite-dimensional Kac--Moody superalgebras, which are not Lie algebras, are the orthosymplectic Lie superalgebras $osp(1|2n)$. Naturally, it is an important task to construct all the irreducible representations of the spin quiver Hecke algebras corresponding to $osp(1|2n)$.

In this paper, we construct all the irreducible representations of spin quiver Hecke algebras for orthosymplectic Lie superalgebras $osp(1|2n)$. Our method is similar to that of Kleshchev and Ram \cite{KR} and is based on the work of Clark, Hill, and Wang \cite{CHW} on quantum shuffles and dominant Lyndon words. Both of these papers are closely related in the work of Leclerc \cite{Lec}. We present an explicit construction of cuspidal representations in Proposition \ref{prop-cusp} and use the cuspidal representations as building blocks to obtain other irreducible representations. In this process, the computation of the leading coefficients of canonical basis elements is crucial and requires a careful analysis of signs and degrees for the corresponding representations of the spin quiver Hecke algebra in categorification. With cuspidal representations at hand,  we construct standard representations through induction from cuspidal representations,  and show that they have irreducible heads. Finally, as the main result  (Theorem \ref{thm-main}) of this paper, we prove that these irreducible heads form a complete set of irreducible representations of the spin quiver Hecke algebra for $osp(1|2n)$. 

With the results of this paper, we can consider some future directions. First, as in \cite{BKM2012}, one can use a general convex order to construct standard representations and study their homological properties. Next, one can obtain a concrete crystal structure on the category of representations of $osp(1|2n)$, following \cite{Lauda2011} and \cite{KKO14}. We hope that these directions may be pursed in the near future.

The outline of this paper is as follows. In Section \ref{quantum}, we fix notations for $osp(1|2n)$, consider quantum shuffle products and combinatorics of Lyndon words, and recall  the construction of the canonical basis. In Section \ref{spin-quiver}, spin quiver Hecke algebras are introduced and properties of their representations are presented. The next section is devoted to the construction of cuspidal representations of the spin quiver Hecke algebras. In the last section, we construct standard representations and obtain all the irreducible representations.

\subsection*{Acknowledgments} 
We thank Se-jin Oh for useful comments.


\section{Quantum superalgebras and canonical bases} \label{quantum}

\subsection{Root data}

Let $I=I_{\bar 0} \cup I_{\bar 1}$ be a $\mathbb Z/2 \mathbb Z$-graded finite set of size $n$, and let $p:I \rightarrow \{ \bar 0, \bar 1\}$ be the corresponding parity function.  We assume that $I_{\bar 1} \neq \emptyset$. Consider a generalized Cartan matrix $A=(a_{ij})_{i,j \in I}$ such that (C1) $a_{ii}=2$ for each $i \in I$; (C2) $a_{ij} \in \mathbb Z_{\le 0}$ for $i \neq j$; (C3) $a_{ij} =0$ if and only if $a_{ji}=0$; (C4) $a_{ij} \in 2\mathbb Z$ for $i\in I_{\bar 1}$ and $j \in I$.
We assume that the matrix $A$ is symmetrizable, i.e. there exists an invertible matrix $D=\mathrm{diag}(s_1, \ldots , s_n)$ with $DA$ symmetric. Furthermore, we choose $D$ such that $s_i \in \mathbb Z_{>0}$ and $\mathrm{gcd}(s_1, \ldots , s_n)=1$, and  assume that 
the integer $s_i$ is odd if and only if $i \in I_{\bar 1}$. 

In this paper, we will be primarily interested in the following case:  the
index set $I=\{ 1, 2,  \ldots , n \}$ with $I_{\bar 1} = \{ n \}$,
\begin{equation} \label{eqn-ma} \scriptsize{A= \begin{pmatrix} 2 & -1 & 0
&\cdots &  0 & 0 &0  \\ -1 & 2 & -1 &  \cdots & 0 &0&0  \\ 0 & -1 & 2 &
\cdots & 0 &0&0 \\  \vdots & \vdots & \vdots & \vdots & \vdots & \vdots &
\vdots \\ 0& 0 &0  & \cdots &  -1 & 2 & -1  \\ 0&0&0 & \cdots &  0 & -2 & 2
\end{pmatrix} },  \end{equation}  and $D= (2, 2, \ldots , 2, 1)$. Throughout this paper, we let
$\mathfrak g$ be the Kac-Moody superalgebra  associated to  a symmetrizable
generalized Cartan matrix $A$ as in (\ref{eqn-ma}) and let $U_q(\mathfrak g)$ be the corresponding
quantized enveloping superalgebra defined as in \cite{BKM}. The generators of
$\mathfrak g$ will be denoted by $e_i, f_i$ and $h_i$ $(i \in I)$. The
subalgebra of $U_q(\mathfrak g)$ generated by the elements $e_i$ ($i \in I$)
will be denoted by $U^+_q$.  Let $\widetilde \Phi=\widetilde \Phi_{\bar 0} \cup \widetilde \Phi_{\bar 1}$ be
the root system for $\mathfrak g$ and let \[ \Phi=\Phi_{\bar 0} \cup \Phi_{\bar 1}=\{ \beta \in \widetilde \Phi \ |\  \tfrac 1 2 \beta  \notin \widetilde \Phi \} \] be the reduced root system for $\mathfrak g$, where $\Phi_s=\Phi \cap \widetilde \Phi_s$ for $s \in \{ \bar 0, \bar 1 \}$.
Denote the set of simple roots by $\Pi =\Pi_{\bar 0}
\cup \Pi_{\bar 1}=\{\alpha_i | i \in I \}$ and the set of positive roots by $\widetilde \Phi^+$. Then we put $\Phi^+=\Phi \cap \widetilde \Phi^+$. We also have 
the corresponding sets $\widetilde \Phi_{\bar 0}^+, \Phi_{\bar 0}^+$ (resp. $\widetilde \Phi_{\bar 1}^+, \Phi_{\bar 1}^+$)  of positive even (resp. odd) roots.
For example, when $n=2$, we have $I_{\bar 1}= \{ 2 \}$ and 
\[ \begin{array}{ll} \widetilde \Phi^+ =\{ \alpha_1,\ \alpha_2,\ \alpha_1+\alpha_2,\ \alpha_1 +2 \alpha_2,\ 2 \alpha_2 ,\ 2\alpha_1+2 \alpha_2 \}, & \Phi^+ =\{ \alpha_1,\ \alpha_2,\ \alpha_1+\alpha_2,\ \alpha_1 +2 \alpha_2 \}, \\ \widetilde \Phi^+_{\bar 0} =\{ \alpha_1,\  \alpha_1 +2 \alpha_2,\ 2 \alpha_2 ,\ 2\alpha_1+2 \alpha_2 \}, & \Phi^+_{\bar 0} =\{ \alpha_1,\ \alpha_1 +2 \alpha_2 \}, \\ \widetilde \Phi^+_{\bar 1} = \Phi^+_{\bar 1} =\{ \alpha_2,\ \alpha_1+\alpha_2 \}. & \end{array} \]

The $\mathbb Z$-lattice spanned by $\Pi$
is denoted by $Q$. We define $p(\alpha_i)=p(i)$, $i \in I$,  and extend it to
the additive monoid $Q^+:= \sum_i \mathbb Z_{\ge 0} \alpha_i$. Define a
symmetric bilinear form $(\cdot, \cdot):Q \times Q \longrightarrow \mathbb Z$
by $(\alpha_i,\alpha_j)=b_{ij}$, where $B=DA=(b_{ij})$.

\subsection{Quantum shuffle superalgebras}

Let $\mathcal W$ be the set of words on the alphabet $I$ with the empty word $\emptyset$. An element $\mathbf i \in \mathcal W$ will be denoted by \[ \mathbf i = (i_1, i_2, \dots, i_d)=i_1i_2\dots i_d.\]   
Define $|\mathbf i |=|(i_1, \dots, i_d)|=\alpha_{i_1}+ \cdots + \alpha_{i_d} \in Q^+$ and $p(\mathbf i) =p(|\mathbf i|)$ for $\mathbf i \in \mathcal W$. The length of $\bi$ will be denoted by $\ell(\bi)$, i.e. $\ell(i_1, i_2, \dots , i_d)=d$. For $\alpha \in Q^+$, set $\mathcal W_\alpha=\{ \mathbf i \in \mathcal W \, | \, | \mathbf i | = \alpha \}$. 
Let $\mathcal F$ be the free associative superalgebra over $\mathbb Q(q)$ generated by $I$, where $q$ is an indeterminate. Note that $\mathcal F$ has a weight decomposition $\displaystyle\mathcal F=\bigoplus_{\alpha\in Q^+}\mathcal F_{\alpha},$ where $\mathcal F_{\alpha}=\mathcal F\cap \mathcal W_{\alpha}.$ The set $\mathcal W$ is naturally considered as a $\mathbb Q(q)$-linear basis of $\mathcal F$. 

We define the quantum shuffle product $\diamond:\mathcal F \times \mathcal F \longrightarrow \mathcal F$ inductively by
\begin{equation} \label{ind} (xi)\diamond (yj)=(x \diamond (yj))i +(-1)^{p(xi)p(j)}q^{-(|xi|,|j|)}((xi)\diamond y)j \end{equation}
for $x, y \in \mathcal W$ and $i,j \in I$ and by extending it linearly, where we set $\emptyset \diamond x=x \diamond \emptyset=x$ for $x \in \mathcal W$.


\begin{Prop} \cite[Corollary 3.4]{CHW}
There exists an algebra embedding $\Psi:U^+_q \longrightarrow (\mathcal F, \diamond)$ such that $\Psi(e_i) =i$.
\end{Prop}

Define $\mathcal U = \Psi(U^+_q)$ to be the subalgebra of $(\mathcal F, \diamond)$ generated by $I$. The algebra $\mathcal U$ is $Q^+$-graded with $\mathcal U_{\alpha}=\mathcal U\cap \mathcal F_{\alpha}.$
We define the shuffle product on $\mathcal F \otimes \mathcal F$ by 
\[ (w \otimes x) \diamond (y \otimes z) = (-1)^{p(x)p(y)} q^{-(|x|,|y|)}(w \diamond y) \otimes (x \diamond z) \quad \text{ for } x,y,z,w \in \mathcal W, \] and
define the map $\Delta: \mathcal F  \longrightarrow \mathcal F \otimes \mathcal F$ by 
\[ \Delta(i_1, \dots, i_d) = \sum_{0 \le k \le d} (i_{k+1}, \dots , i_d) \otimes (i_1, \dots, i_k).  \]
\begin{Prop}\cite[Proposition 3.13]{CHW} \label{prop-bi}
There exists a nondegenerate symmetric bilinear form
\[ (\cdot , \cdot) : \mathcal U \times \mathcal U \longrightarrow \mathbb Q(q) \] that satisfies the following properties:
\begin{enumerate}
\item $(1,1)=1$;
\item $(i,j)=\delta_{ij}$ for $i,j \in I$;
\item $(x, y \diamond z) = (\Delta(x), y \otimes z)$ for $x, y, z \in \mathcal U$, 
where the induced bilinear form is given by
\[ (x \otimes x', y \otimes y') := (x, y)(x',y') .\]
\end{enumerate}
\end{Prop}

In the following proposition, we recall some linear maps on $\mathcal F,$ which give rise to important (anti-)automorphisms on $\mathcal U.$ For $\nu = \sum_i c_i \alpha_i \in Q^+$, let 
\begin{equation}\label{eqn-definition of N and P}
 N(\nu) = \frac 1 2 ( (\nu, \nu) - \sum_{i\in I} c_i (\alpha_i, \alpha_i) )\quad \text{and} \quad P(\nu)=\frac 1 2(p(\nu)^2-\sum_{i\in I}c_ip(\alpha_{i})),\end{equation}
 where $p(\alpha_i)\in\{0,1\}$ and $\displaystyle\sum_{i\in I}c_ip(\alpha_{i})$ are interpreted as integers. For any $\bi\in \mathcal W,$ we set $P(\bi)=P(|\bi|).$ 

\begin{Prop}[{\cite[Proposition 3.10]{CHW}}] \label{prop-linear maps}\hfill
\begin{enumerate}
\item Let $\tau: \mathcal F\to\mathcal F$ be the $\mathbb{Q}(q)$-linear map defined by \[\tau(i_1,\ldots,i_d)=(i_d,\ldots,i_1).\] Then $\tau(x\diamond y)=\tau(y)\diamond \tau(x)$ for all $x,y\in\mathcal F.$
\item Let $\overline{\phantom{L}}: \mathcal F \to \mathcal F$ be the  the $\mathbb Q$-linear map defined by $\overline{q}=-q^{-1}$ and \[\overline{(i_1,\ldots,i_d)}=(-1)^{\sum_{s<t}p(i_s)p(i_t)}q^{-\sum_{s<t} (\alpha_{i_s},\alpha_{i_t})}(i_d,\ldots,i_1).\] Then  $\overline{x\diamond y}=\overline{x}\diamond \overline{y}$ for all $x,y\in\mathcal F.$
\item Let $\sigma: \mathcal F\to\mathcal F$ be the $\mathbb Q$-linear map defined by $\sigma(q)=-q^{-1}$ and \[\sigma(i_1,\ldots,i_d)=(-1)^{\sum_{s<t}p(i_s)p(i_t)}q^{-\sum _{s<t}(\alpha_{i_s},\alpha_{i_t})}(i_1,\ldots,i_d).\]  Then $\sigma(x)=\overline{\tau(x)}$,  $\sigma(x\diamond y)=\sigma(y)\diamond \sigma(x)$ for all $x,y\in\mathcal F,$ and  \[\sigma(\bi)=(-1)^{P(\bi)}q^{-N(|\bi|)}\bi  \; \text{ for all}\; \bi\in \mathcal W .\]

\end{enumerate}
\end{Prop}

\begin{Rmk} Since $\displaystyle\sum_{s<t}(\alpha_{i_s},\alpha_{i_t})\in2\mathbb{Z},$ it is easy to check from the definition that $\sigma^2=\mathrm{Id}_{\mathcal F}.$ 
\end{Rmk}

The following lemma will be useful later.

\begin{Lem}\label{lem-N(nu)} Let $\mu,\nu\in Q^+$ and $n\in\mathbb Z_{>0}.$ We have the following properties
\begin{enumerate}
\item $N(\nu)\in2\mathbb{Z}$  for all $\nu\in Q^+;$
\item $N(\mu+\nu)=N(\mu)+N(\nu)+(\mu,\nu);$
\item $P(\mu+\nu)=P(\mu)+P(\nu)+p(\mu)p(\nu)$.
\end{enumerate}
\end{Lem}

\begin{proof}

It is is easy to see from  (\ref{eqn-definition of N and P}) that $\displaystyle N(\alpha_{i_1}+\ldots+\alpha_{i_{k}})=\sum_{1\leq s<t\leq k}(\alpha_{i_s},\alpha_{i_t}).$ Now  statement (1)  follows from the fact that  $(\alpha_i,\alpha_{j})\in 2\mathbb{Z}$ for all $i,j\in I.$ The equalities (2) and (3) follow from (\ref{eqn-definition of N and P}) by straightforward computations. 
\end{proof}

\subsection{Dominant words and Lyndon words}

Fix a total ordering $\prec$ on $I$ to be $1 \prec 2 \prec \dots \prec n$, and put the induced lexicographic ordering $\prec$ on $\mathcal W$. A word $\mathbf i \in \mathcal W$ is called {\em dominant} if $\mathbf i =\max(u)$ for some $u \in \mathcal U$. Denote the set of dominant words by $\mathcal W^+$, and define $\mathcal W^+_\alpha = \mathcal W^+ \cap \mathcal W_\alpha$. 
A word $\mathbf i =(i_1, \dots , i_d) \in \mathcal W$ is called {\em Lyndon} if it is smaller than any of its proper right factors. Let $\mathcal L$ be the set of Lyndon words in $\mathcal W$, and let $\mathcal L^+$ be the set of dominant Lyndon words in $\mathcal W$. Recall that every word $\mathbf i \in \mathcal W$ has a {\em canonical factorization} as a product of non-increasing Lyndon words:
\[ \mathbf i = \mathbf i_1 \cdots \mathbf i_d, \quad \mathbf i _1, \dots, \mathbf i _d \in \mathcal L, \quad  \mathbf i_1 \succeq \cdots \succeq \mathbf i_d .\]

\begin{Thm}\cite[Theorem 4.8]{CHW} \label{thm-Lyndon basis}\hfill
\begin{enumerate}
\item The map $\mathbf i \mapsto | \mathbf i |$ is a bijection from $\mathcal L^+$ to $\Phi^+$. Given $\beta \in\Phi^{+},$ we  write $\boi^{+}(\beta)$ for the pre-image of $\beta$ under this bijection.  

\item Assume that  $\mathbf i =\mathbf i_1 \cdots \mathbf i_d$ is the canonical factorization. Then  $\mathbf i \in \mathcal W^+$ if and only if
$\mathbf i_s \in \mathcal L^+$ for each $s =1,2, \dots , d$. 
\end{enumerate}
\end{Thm}

The set of dominant Lyndon words was computed in the work of Clark, Hill and Wang: 

\begin{Prop}\cite[Proposition 6.5]{CHW} \label{pro-lyn}
The set of dominant Lyndon words for $\mathfrak g$ is given by
\[\mathcal L^+ = \{ (i, \dots, j) | 1 \le i \le j \le n\} \cup \{(i, \dots , n, n, \dots, j)| 1 \le i < j \le  n \}. \]
\end{Prop}

\begin{Rmk}
As a related result, a basis for $\mathfrak g$ arising from Lyndon words was obtained by Bokut, Kang, Lee and Malcomson in \cite{BKLM}. 
\end{Rmk}

The following corollary is similar to \cite[Lemma 5.9]{KR} and  slightly generalizes  \cite[Corollary 4.17]{CHW} in our context. 

\begin{Cor}\label{cor-Lyndon word power}
Let $\beta\in\Phi^+$ and $m\in\mathbb{Z}_{\geq0}.$ Then $\boi^+(\beta)^m$ is the smallest dominant word in $\mathcal W_{m\beta}.$ 
\end{Cor}

\begin{proof}
 Let $\bi=\boi^+(\beta),$ and let $\bj$ be a dominant word of weight $m\beta$ such that $\bj\prec \bi^m.$ We show that this is impossible by  checking the different cases for $\bi\in\mathcal L^+.$ Let $\bj=\bj_1\bj_2\ldots\bj_s$ be the canonical factorization of $\bj,$ where $\bj_1,\bj_2,\ldots,\bj_s\in\mathcal L^+$ and $\bj_1\succeq\bj_2\succeq	\ldots\succeq\bj_s.$  Since $\bj\prec\bi^m,$  there exists $k$ such that $\bj_r=\bi$ for $r<k$ and $\bj_{k}\prec\bi.$ (See \cite{Melancon1992}.)    Clearly $k\leq m.$ Suppose $\bj=\bi^{k-1}\bj_{k}\ldots\bj_s,$ where $\bi\succ\bj_{k}\succeq\ldots\succeq\bj_s.$  Set $\gamma_s=|\bj_s|$ for all $s.$ By assumption $\gamma_1+\ldots+\gamma_s=m\beta.$ Recall that $\mathcal L^+ = \{ (i, \dots, j) | 1 \le i \le j \le n\} \cup \{(i, \dots , n, n, \dots, j)| 1 \le i < j \le  n \}.$ Assume that $\bi=(i, \dots, j)$  for some $1\leq i\leq j\leq n.$  Thus $m\beta=m\alpha_i+\ldots+m\alpha_j.$ Since $\bj_r\prec\bi$ for  $k\leq r\leq s,$ and the coefficient of $\alpha_{\ell}$ for $\ell<i$ in  $m\beta$ is zero,   it follows that $\bj_r\in\{(i,\ldots,t)| i \leq  t\leq j-1\}$ for all $k\leq r\leq s,$ and from the coefficient of $\alpha_i$ in $m\beta$ we conclude that $s=m.$ But then  the coefficient of $\alpha_j$ in $\gamma_1+\ldots+\gamma_s$ will be $k-1<m$, which is a contradiction. Next, suppose that $\bi=(i, \dots, n,n)$  for some $1\leq i\leq n-1.$  Then  $m\beta=m\alpha_i+\ldots+2m\alpha_n.$ By a similar argument  it follows that $\bj_r\in\{(i,\ldots,t)| i \leq t\leq n\}$ for $k\leq r\leq s,$ $s=m,$ and the coefficient of $\alpha_n$ in $\gamma_1+\ldots+\gamma_s$ is  $2(k-1)<2m$, which is again contradiction. Finally, assume that $\bi=(i, \dots, n,n,\ldots,j)$  for some $1\leq i<j \leq n-1.$  Then  $m\beta=m\alpha_i+\ldots+m\alpha_{j-1}+2m\alpha_j+\ldots+2m\alpha_n.$ Similarly as above it follows that $\bj_r\in\{(i,\ldots,t)| i\leq t\leq n\}\cup\{(i,\ldots,n,n,\dots, t)| i<t\leq j+1\},$ for $k\leq r\leq s,$ $s=m$ and the coefficient of $\alpha_j$ in $\gamma_1+\ldots+\gamma_s$ will be at most $2(k-1)+m-k+1=m+k-1<2m$, which is another contradiction.
\end{proof}

\subsection{Maximal elements in shuffle products}

Let $\mathcal A=\mathbb Z[q,q^{-1}]$.
For $\mathbf i\in\mathcal L^+$ set $q_{\bi}:=q^{\frac{(|\bi|,|\bi|)}{2}},$ and define \begin{equation}\label{eq-quantum numbers}[m]_{\mathbf i}= \begin{cases} \displaystyle{\frac {q_{\bi}^{m}-q_{\bi}^{-m}}{q_{\bi}-q_{\bi}^{-1}} } & \text{ if } p(\mathbf i)=\bar 0, \\ \displaystyle{ \frac {(-q_{\bi})^{m}-q_{\bi}^{-m}}{-q_{\bi}-q_{\bi}^{-1}}} & \text{ if } p(\mathbf i)=\bar 1
\end{cases}  
\quad \text{and}\quad  [m]_{\bi}!=[m]_{\bi}[m-1]_{\bi}\ldots[1]_{\bi}.
\end{equation} 
In particular, $[2]_n=-(q-q^{-1})$,
\[ q_{\mathbf i}= \begin{cases} \displaystyle{q^2} & \text{ if } p(\mathbf i)=\bar 0, \\ \displaystyle{ q} & \text{ if } p(\mathbf i)=\bar 1,
\end{cases}
\quad \text{and} \quad [m]_{\mathbf i}= \begin{cases} \displaystyle{\frac {q^{2m}-q^{-2m}}{q^2-q^{-2}} } & \text{ if } p(\mathbf i)=\bar 0, \\ \displaystyle{ \frac {(-q)^{m}-q^{-m}}{-q-q^{-1}}} & \text{ if } p(\mathbf i)=\bar 1.
\end{cases}
\]

The  following lemma follows as in \cite[Lemma 5.1]{KR}.

\begin{Lem}\label{lem-shuffle product order}Let $w,w',\ell,g\in\mathcal W$ with $|w|=|\ell|,$ $|w'|=|g|,$ $w\preceq \ell$ and $w'\preceq g.$ Then $\max(w\diamond w')\preceq \max (l\diamond g).$
Moreover, if $w\prec \ell$ or $w'\prec g,$ then $\max(w\diamond w')\prec \max (l\diamond g).$
 \end{Lem}
 



 

The next result generalizes   \cite[Lemma 4.5]{CHW} and will be useful in computing leading coefficients in quantum shuffle products for canonical factorizations.

\begin{Lem}\label{lem-top coeff special} Assume that $\bi\in\mathcal L,$  $\bj\in\mathcal W^+,$ $\bi\succeq \bj$ and $n\in\mathbb Z_{>0}.$ Then $\mathrm{max}(\bi^{n}\diamond \bj)=\mathrm{max}(\bj\diamond \bi^n)=\bi^{n}\bj.$ Moreover, 
\begin{enumerate}
\item If $\bi\succ\bj,$ then  the coefficient of $\bi^{n}\bj$ in $\bi^{n}\diamond \bj$ is   $(-1)^{p(n\bi)p(\bj)}q^{-(n|\bi|,|\bj|)}.$ 
\item If $\bi\succ\bj,$ then  the coefficient of $\bi^{n}\bj$ in $\bj\diamond \bi^{n}$ is $1.$
\item The coefficient of $\bi^{n+1}$ in  $\bi^{n}\diamond \bi$ is $(-1)^{p(n\bi)}q_{\bi}^{-n}[n+1]_{\bi}.$
\end{enumerate}
\end{Lem}

\begin{proof}
Let $\bi\succ\bj.$  We will prove that $\mathrm{max}(\bi^{n}\diamond \bj)=\bi^{n}\bj$ and (1) by   induction on $n.$  The case $n=1$ follows from \cite[Lemma 4.5]{CHW}. Assume that $\mathrm{max}(\bi^{n-1}\diamond \bj)=\bi^{n-1}\bj$ for all $\bj\in\mathcal W^+$ such that $\bi\succ\bj$. Suppose that the word  $\mathbf k$ occurs as a nontrivial shuffle in $\bi^n\diamond \bj$ (i.e. $\mathbf k\neq \bi^n\bj$.) Then there exists a factorization $\bj=\bj_1\bj_2$ such that $\mathbf k$ occurs in $(\bi^{n-1}\diamond \bj_1)(\bi\diamond \bj_2).$ Clearly $\bi\succ\bj\succeq\bj_1$ and $\bj_1\in\mathcal W^+$ (as a factor of a dominant word). Hence by the inductive assumption 
\[\mathbf k\preceq \bi^{n-1}\bj_1\max(\bi\diamond \bj_2). \]  Now, since any word occurring in $\bj_1(\bi\diamond \bj_2)$ is a proper shuffle in $\bi\diamond (\bj_1\bj_2)=\bi\diamond \bj$ and the maximum word in $\bi\diamond \bj$ is $\bi\bj$ we have 
\[\mathbf k\preceq \bi^{n-1}\bj_1\max(\bi\diamond \bj_2)\prec \bi^{n-1}\mathrm{max}(\bi\diamond \bj)=\bi^n\bj\]  which proves that $\mathrm{max}(\bi^{n}\diamond \bj)=\bi^n\bj$. 

Next, we prove that the coefficient of $\bi^{n}\bj$ in $\bi^{n}\diamond \bj$ is $(-1)^{p(n\bi)p(\bj)}q^{-(n|\bi|,|\bj|)}$ by induction on $\ell(\bj).$ Let $\bi=(i_1,\ldots,i_d),$ $\bj=(j_1,\ldots,j_k)$ and assume that $\bi\succ\bj.$  Suppose that  $\ell(\bj)=1,$ so $\bj=j_1=j\in I$ and $\bi>j.$ Then we have $j<i_1.$ We claim that   the coefficient of $\bi^nj$ in $\bi^n\diamond j$ is $(-1)^{p(n\bi)p(j)}q^{-(|n\bi|,\alpha_{j})}.$  We have
\begin{align*}
\bi^{n}\diamond j&=\bi^{n-1}(i_1,\ldots,i_d)\diamond j\\
&=(\bi^{n-1}(i_1,\ldots,i_{d-1})\diamond j)i_d+(-1)^{p(n\bi)p(j)}q^{-(|n\bi|,\alpha_{j})}\bi^nj.
\end{align*}
We claim that $\max((\bi^{n-1}(i_1,\ldots,i_{d-1})\diamond j)i_d)\prec \bi^nj.$ Indeed, $\bi^{n-1}(i_1,\ldots,i_{d-1})ji_d\prec \bi^nj$ and any nontrivial shuffle  in $\bi^{n-1}(i_1,\ldots,i_{d-1})\diamond j$ occurs as a shuffle in either $(\bi^{n-1}\diamond j)(i_1,\ldots,i_{d-1})$ or in $\bi^{n-1}((i_1,\ldots,i_{d-1})\diamond j).$ By the above and Lemma \ref{lem-shuffle product order} we have that $\max((\bi^{n-1}\diamond j)(i_1,\ldots,i_{d-1}))i_d\prec\bi^nj$ and  $\max(\bi^{n-1}((i_1,\ldots,i_{d-1})\diamond j))i_d\prec\bi^nj$. Therefore, $\mathrm{max}((\bi^{n-1}(i_1,\ldots,i_{d-1})\diamond j)i_d)\prec\bi^{n}j$ which proves that  the coefficient of $\bi^nj$ in $\bi^n\diamond j$ has to be $(-1)^{p(n\bi)p(j)}q^{-(|n\bi|,\alpha_{j})}.$

For the inductive step, assume that the coefficient of $\bi^{n}\mathbf w$ is $(-1)^{p(n\bi)p(\mathbf w)}q^{-(n|\bi|,|\mathbf w|)}$ for all $n$ and $\mathbf w\in\mathcal W^{+}$ such that $\bi\succ \mathbf w$ and $\ell(\mathbf w)<\ell(\bj).$ We have
\begin{align}\label{lem-top coeff special-eqn 1}
\bi^{n}\diamond \bj&=\bi^{n-1}(i_1,\ldots,i_d)\diamond (j_1,\ldots,j_k)\notag\\
&=(\bi^{n-1}(i_1,\ldots,i_{d-1})\diamond \bj)i_d+(-1)^{p(n\bi)p(j_{k})}q^{-(|n\bi|,\alpha_{j_{k}})}(\bi^n\diamond (j_1,\ldots,j_{k-1}))j_k.
\end{align} 

Again we show that $\max((\bi^{n-1}(i_1,\ldots,i_{d-1})\diamond \bj)i_d)\prec \bi^n\bj.$ Clearly,  $\bi^{n-1}(i_1,\ldots,i_{d-1})\bj i_d\prec \bi^n\bj,$ and if $\mathbf k$ is any nontrivial shuffle  in $\bi^{n-1}(i_1,\ldots,i_{d-1})\diamond \bj,$ then there exists a factorization $\bj=\bj_1\bj_2$ such that $\mathbf k$ occurs in $(\bi^{n-1}\diamond \bj_1)((i_1,\ldots,i_{d-1})\diamond \bj_2).$  Again $\bi\succ\bj\succeq\bj_1$ and $\bj_1\in\mathcal W^+$. Since 
$\max(\bi^{n-1}\diamond \bj_1)=\bi^{n-1}\bj_1,$ it follows that 
\[\mathbf k\preceq \bi^{n-1}\bj_1\max((i_1,\ldots,i_{d-1})\diamond \bj_2). \]  Now, since any word occurring in $\bj_1((i_1,\ldots,i_{d-1})\diamond \bj_2)$ is a proper shuffle in $(i_1,\ldots,i_{d-1})\diamond (\bj_1\bj_2)=(i_1,\ldots,i_{d-1})\diamond \bj,$ we have
\[\mathbf k\prec \bi^{n-1}\mathrm{max}((i_1,\ldots,i_{d-1})\diamond \bj)\prec\bi^{n-1}\ \mathrm{max}(\bi\diamond \bj)=\bi^n\bj.\]
By induction on $\ell(\bj)$ and  \eqref{lem-top coeff special-eqn 1}, it follows that the coefficient of $\bi^n\bj$ in $\bi^n\diamond \bj$  is $(-1)^{p(n\bi)p(\bj)}q^{-(|n\bi|,|\bj|)},$ which proves (1).

Next, we prove that $\max(\bj\diamond \bi^{n})=\bi^n\bj.$ By Proposition \ref{prop-linear maps}, we have $\sigma(\bj)=(-1)^{P(\bj)}q^{-N(|\bj|)}\bj$  and $\sigma(\bi^{n})=(-1)^{P(n\bi)}q^{-N(n|\bi|)}\bi^{n}.$ Since $\sigma$ is an anti-automorphism, we have
\begin{align}\label{lem-top coeff special-eqn 2}
\sigma(\bj\diamond \bi^{n})&=\sigma(\bi)^{n}\diamond \sigma(\bj)=(-1)^{P(n\bi)+P(\bj)}q^{-(N(n|\bi|)+N(|\bj|))}\bi^{n}\diamond \bj .
\end{align}

By  Proposition \ref{prop-linear maps}, $\mathrm{max}(\sigma(u))=\mathrm{max}(u)$ for all $u\in\mathcal U.$ Therefore by \eqref{lem-top coeff special-eqn 2} we have \\
$\mathrm{max}(\bj\diamond \bi^{n})=\mathrm{max}(\bi^{n}\diamond \bj)=\bi^n\bj$. Moreover,  the coefficient of $\bi^{n}\bj$ in $\bi^{n}\diamond \bj$ is $(-1)^{p(n\bi)p(\bj)}q^{-(n|\bi|,|\bj|)}$  by (1). Lemma \ref{lem-N(nu)} yields  
\[
N(n|\bi|)+N(|\bj|)+(|n\bi|,|\bj|)=N(|\bi^n\bj|) \quad \text{ and } \quad 
P(n\bi)+P(\bj)+p(n\bi)p(\bj)=P(\bi^n\bj).
\]
Consequently,  for some coefficients $a_{\mathbf k}\in\mathcal A,$ we have
\begin{equation}\label{eq-part 2}
\sigma(\bj\diamond \bi^{n})=(-1)^{P(\bi^n\bj)}q^{-N(|\bi^n\bj|)}\bi^n\bj+\sum_{\mathbf k\prec \bi^n\bj}a_{\mathbf k}\mathbf k.
\end{equation} Since $\sigma^2=\mathrm{Id}_{\mathcal F}$,  $\sigma(q)=-q^{-1},$  $\sigma(\bi^n\bj)=(-1)^{P(\bi^n\bj)}q^{-N(|\bi^n\bj|)}\bi^{n}\bj$ and by Lemma \ref{lem-N(nu)}(1) $N(|\bi^n\bj|)$ is even,   it follows  by (\ref{eq-part 2}) that
\begin{align*}
\bj\diamond \bi^{ n}&=\bi^n\bj+\sum_{\sigma(\mathbf k)\prec \bi^n\bj}\sigma(a_{\mathbf k})\mathbf \sigma(\mathbf k).
\end{align*} This proves (2). 

 Next, we prove that $\max(\bi^{n}\diamond \bi)=\bi^{n+1}$ and $(3)$ by induction on $n.$ Let $\bi=(i_1,\ldots,i_d).$  Suppose that $n=1.$ By \cite[Lemma 4.5]{CHW},  $\max(\bi\diamond\bi)=\bi^2$ and the coefficient of $\bi^2$ in $\bi\diamond \bi$ is $$1+(-1)^{p(\bi)}q^{-(|\bi|,|\bi|)}=1+(-1)^{p(\bi)}q_{\bi}^{-2}=(-1)^{p(\bi)}q_{\bi}^{-1}[2]_{\bi}.$$   For the inductive step, assume that $\max(\bi^{n-1}\diamond \bi)=\bi^{n}$ and that the coefficient of $\bi^{n}$ in $\bi^{n-1}\diamond \bi$ is $(-1)^{(n-1)p(\bi)}q_{\bi}^{-(n-1)}[n]_{\bi}.$ If $d=1,$ then the result follows easily by induction on $n$ and the fact that \begin{equation}\label{lem-top coeff special-eqn 4}(-1)^{(n-1)p(\bi)}q_{\bi}^{-(n-1)}[n]_{\bi}+(-1)^{p(n\bi)p(\bi)}q^{-(n|\bi|,|\bi|)}=(-1)^{np(\bi)}q_{\bi}^{-n}[n+1]_{\bi},\end{equation} which can be easily verified from (\ref{eq-quantum numbers}). Assume that $d>1$ and let $\bj=(i_1,\ldots, i_{d-1}).$ Then $\bj\in\mathcal W^+$ and $\bi\succ \bj.$ Moreover,
\begin{align*}
\bi^{n}\diamond \bi&=(\bi^{n-1}\bj\diamond \bi)i_d+(-1)^{p(n\bi)p(i_d)}q^{-(n|\bi|,|i_d|)}(\bi^n\diamond \bj)i_d.
\end{align*}
By part (2)  there exist $b_{\mathbf k}\in\mathcal A$ such that
\begin{equation}\label{lem-top coeff special-eqn 3}
\bi^{n-1}\bj=\bj\diamond \bi^{n-1}+\sum_{\mathbf k\prec \bi^{n-1}\bj}b_{\mathbf k}\mathbf k.
\end{equation} 

\noindent Hence, using  \eqref{lem-top coeff special-eqn 3}, induction on $n,$ (1), and (2), 
we obtain for some $a_{\mathbf h}, c_{\mathbf s}\in \mathcal A$
\begin{align*}
\bi^{n}\diamond \bi&=((\bj\diamond \bi^{n-1})\diamond \bi)i_d+\sum_{\mathbf k\prec \bi^{n-1}\bj}b_{\mathbf k}(\mathbf  k\diamond \bi)i_d+(-1)^{p(n\bi)p(i_d)}q^{-(n|\bi|,|i_{d}|)}(\bi^n\diamond \bj)i_d\\
&=(\bj\diamond (\bi^{n-1}\diamond\bi))i_d+(-1)^{p(n\bi)p(i_d)}q^{-(n|\bi|,|i_{d}|)}(\bi^n\diamond \bj)i_d+\sum_{\mathbf k\prec \bi^{n-1}\bj}b_{\mathbf k}(\mathbf k\diamond \bi)i_d\\
&=(-1)^{(n-1)p(\bi)}q_{\bi}^{-(n-1)}[n]_{\bi}(\bj\diamond \bi^{n})i_d+(-1)^{p(n\bi)p(|i_d|+|\bj|)}q^{-(n|\bi|,|i_d|+|\bj|)}(\bi^n\bj)i_d\\
& \qquad +\sum_{\mathbf h\prec \bi^n\bj}a_{\mathbf h}\mathbf hi_d+\sum_{\mathbf k\prec \bi^{n-1}\bj}b_{\mathbf k}(\mathbf k\diamond \bi)i_d\\
&=\left((-1)^{p((n-1)\bi)}q_{\bi}^{-(n-1)}[n]_{\bi}+(-1)^{p(n\bi)p(\bi)}q^{-(n|\bi|,|\bi|)}\right)\bi^{n+1} \\
& \qquad + \sum_{\mathbf s\prec \bi^n\bj}c_{\mathbf s}\mathbf s i_d+\sum_{\mathbf k\prec \bi^{n-1}\bj}b_{\mathbf k} (\mathbf k\diamond \bi)i_d .
\end{align*}

\noindent Moreover,  $\mathbf si_d\prec\bi^{n}\bj i_d=\bi^{n+1}$ for all $\mathbf s\prec \bi^n\bj$ and by Lemma \ref{lem-shuffle product order} we get that $\max(\mathbf k\diamond \bi)i_d\prec\max(\bi^{n-1}\bj\diamond \bi)i_d=(\bi^{n}\bj)i_d=\bi^{n+1}$ for all $\mathbf k\prec \bi^{n-1}\bj.$   Hence  $\max(\bi^{n}\diamond \bi)=\bi^{n+1}$ and the leading coefficient is  $(-1)^{np(\bi)}q_{\bi}^{-n}[n+1]_{\bi}$ by \eqref{lem-top coeff special-eqn 4}.


Finally, $\max(\bi\diamond\bi^{n})=\bi^{n}\diamond \bi$ since $\max(\sigma(u))=\max(u)$ for all $u\in\mathcal U.$

\end{proof}

\begin{Cor}\label{lem-top coeff shuffle special} Let $\bi\in\mathcal L,$  $\bj\in\mathcal W^+,$ $\bi\succ \bj,$ and $n\in\mathbb Z_{>0}.$ Then we have $\max(\bj\diamond \bi^{\diamond n})=\bi^{n}\bj$ and  the coefficient of $\bi^{n}\bj$ in $\bj\diamond \bi^{\diamond n}$ is $(-1)^{\frac{n(n-1)}{2}p(\bi)}q_{\bi}^{-\frac{n(n-1)}{2}}[n]_{\bi}!$.


\end{Cor}

\begin{proof} 
First, we claim that $\max(\bi^{\diamond n})=\bi^{n}$ and that this  top word appears with coefficient \\
$(-1)^{\frac{n(n-1)}{2}p(\bi)}q_{\bi}^{-\frac{n(n-1)}{2}}[n]_{\bi}!.$  The case $n=1$ is trivial. Assume that  $n>1.$ Then by induction, Lemma \ref{lem-top coeff special}(3), and Lemma \ref{lem-shuffle product order}, we have that for some $a_{\mathbf h} \in \mathcal A,$
\begin{align*} \bi^{\diamond n}&=\bi^{\diamond (n-1)}\diamond \bi\\
&=(-1)^{p(\bi)(n-1)(n-2)/2}q_{\bi}^{-(n-1)(n-2)/2}[n-1]_{\bi}!(-1)^{(n-1)p(\bi)}q_{\bi}^{-(n-1)}[n]_{\bi}\bi^{n}+\sum_{\mathbf h\prec \bi^{n}}a_{\mathbf h}\mathbf h\\
&=(-1)^{p(\bi)n(n-1)/2}q_{\bi}^{-n(n-1)/2}[n]_{\bi}!\bi^{n}+\sum_{\mathbf h\prec \bi^{n}}a_{\mathbf h}\mathbf h.
\end{align*} Now the statement of the corollary follows from the above computation and  Lemma \ref{lem-top coeff special}(2).

 \end{proof}

Let $\bi\in\mathcal W^+$ with canonical factorization $\mathbf i= \mathbf i_1^{n_1} \cdots \mathbf i_d^{n_d},$ where $n_1, \dots, n_d \in \mathbb Z_{>0}$, $\mathbf i _1, \dots, \mathbf i _d \in \mathcal L^+$ and $\mathbf i_1 \succ \cdots \succ \mathbf i_d$. 
We define
 \begin{equation} \label{eq-xi-s}
 \xi(\mathbf i)= \sum_{k=1}^d p(\bi_k)n_k (n_k-1)/2 \quad {\rm and}\quad s(\mathbf i)= \sum_{k=1}^d (|\mathbf i_k| , |\mathbf i_k|) n_k (n_k-1)/4.
 \end{equation}

\begin{Cor}\label{cor-top coefficient}
With the notations above, we have, for some $a_{\mathbf k}\in\mathcal A$,
\[\mathbf i_d^{\diamond n_d}\diamond \cdots\diamond \mathbf i_1^{\diamond n_1}=\left((-1)^{\xi(\mathbf i)}q^{-s(\mathbf i)}\prod_{k=1}^d[n_{k}]_{\bi_k}!\right)\bi+\sum_{\mathbf k\prec \bi}a_{\mathbf k}\mathbf k.\]
\end{Cor}

\begin{proof} We will prove the statement by induction on $d.$ If $d=1,$ then $\bi=\bi_1^{n_1}$ and the result follows from Lemma \ref{lem-top coeff shuffle special} (1) since  $\xi(\bi_1^{n_{1}})=p(\bi_1)n_{1}(n_1-1)/2$ and $q^{-s(\bi_1^{n_1})}=q^{-(|\bi_1|,|\bi_1|)n_1(n_{1}-1)/4}=q_{\bi_1}^{-n_1(n_{1}-1)/2}.$

We now proceed to the inductive step. Suppose that $d>1$ and let $\bj=\bi_2^{n_2}\ldots\bi_{d}^{n_d}.$ Then $\bi_1\succ \bj$ and  $\bj\in\mathcal W^+$  by \cite[Lemma 4.2]{CHW}. 
By the inductive hypothesis  we obtain 
\begin{equation}\label{cor- top coefficient eqn1}
\mathbf i_d^{\diamond n_d}\diamond \cdots\diamond \bi_{2}^{\diamond n_2}\diamond \mathbf i_1^{\diamond n_1}=\left((-1)^{\xi(\mathbf j)}q^{-s(\mathbf j)}\prod_{k=2}^d[n_{k}]_{\bi_k}!\right)\bj\diamond \bi_1^{{n_1}}+\sum_{\mathbf{h}\prec \bj}b_{\mathbf{h}}(\mathbf h\diamond \bi_1^{\diamond n_1}).
\end{equation}  
By Lemma \ref{lem-shuffle product order},   $\mathbf h\diamond \bi_{1}^{\diamond n_1}\prec \bj \diamond\bi_{1}^{\diamond n_1}$ for all $\mathbf h\in\mathcal W$ such that $\mathbf h\prec \bj,$ $|\mathbf h|=|\bj|.$ Moreover, by Lemma \ref{lem-top coeff shuffle special}(2), $\mathrm{max}(\bj\diamond \bi^{\diamond n_1})=\bi_{1}^{n_{1}}\bj=\bi$ and the coefficient of $\bi_1^{n_1}\bj$ in $\bj\diamond \bi^{\diamond n_1}$ is $$(-1)^{p(\bi_1)n_1(n_{1}-1)/2}q_{\bi_1}^{-n_1(n_{1}-1)/2}[n_1]_{\bi_1}!=(-1)^{\xi(\bi_1^{n_1})}q^{-s(\bi_1^{n_1})}.$$ The statement of the corollary now follows from \eqref{cor- top coefficient eqn1} and the equalities $\xi(\bi)=\xi(\bi_1^{n_1})+\xi(\bj)$ and $s(\bi)=s(\bi_1^{n_{1}})+s(\bj).$ 
\end{proof}

\subsection{PBW and dual canonical bases}\label{PBW and dual canonical bases}
  For $1\le i \le j \le n$, we set
\begin{enumerate}
\item $E_{\mathbf i}=\begin{cases} (-1)^{j-i} (q^2-q^{-2})^{j-i}q^{-N(|\mathbf i|)}\mathbf i & \text{if } \mathbf i =(i, \dots, j),\\
 (-1)^{j-i} (q^2-q^{-2})^{2n-i-j}q^{-N(|\mathbf i|)}[2]_n^{-1}\mathbf i & \text{if } \mathbf i=(i, \dots, n, n, \dots, j), \ i<j;
\end{cases}$	

\item $E_{\mathbf i}^*=\begin{cases} \mathbf i & \text{if } \mathbf i =(i, \dots, j),\\ [2]_n \mathbf i & \text{if } \mathbf i=(i, \dots, n, n, \dots, j), \ i<j.
\end{cases}$
\end{enumerate}

Let  $\mathbf i  \in \mathcal W^+$. As before, we write the canonical factorization of $\mathbf i$ in the form: \begin{equation} \label{eqn-i}
 \mathbf i= \mathbf i_1^{n_1} \cdots \mathbf i_d^{n_d},\end{equation} where $n_1, \dots, n_d \in \mathbb Z_{>0}$, $\mathbf i _1, \dots, \mathbf i _d \in \mathcal L^+$ and $\mathbf i_1 \succ \cdots \succ \mathbf i_d$. We define
\begin{equation}\label{eqn-Ei}
E_{\mathbf i}=E_{\mathbf i_d}^{(n_d)} \diamond \cdots \diamond E_{\mathbf i_1}^{(n_1)}, 
\end{equation}
where $E_{\mathbf j}^{(m)} = E_{\mathbf j}^{\diamond m}/[m]_{\mathbf j}!$ for $\mathbf j \in \mathcal L^+$, and define
\[ E_{\mathbf i}^* = E_{\mathbf i}/(E_{\mathbf i}, E_{\mathbf i}), \]
where $(\cdot, \cdot)$ is the nondegenerate bilinear form on $\mathcal U$ from  Proposition \ref{prop-bi}. Explicit computations of the bilinear form can be found in \cite[Theorem 5.7]{CHW}. In particular, for $\bi,\bj\in\mathcal W^+$, we have $(E_{\bi},E_{\bj})=0$ unless $\bi=\bj,$ and 
\begin{equation}\label{eqn-(Ei,Ei)}
(E_\bi,E_\bi)=(-1)^{\xi(\bi)}q^{-s(\bi)}\prod_{k=1}^{d}\dfrac{(E_{\bi_k},E_{\bi_k})^{n_k}}{[n_k]_{\bi_k}!}
\end{equation} where $\xi(\bi)$ and $s(\bi)$ are defined in \eqref{eq-xi-s}.

The sets $\{E_{\mathbf i}\,|\, \mathbf i \in \mathcal W^+\}$ and $\{E^*_{\mathbf i}\,|\, \mathbf i \in \mathcal W^+\}$ are bases for $\mathcal U$, called the {\em PBW basis} and the {\em dual PBW basis}, respectively.

\begin{Lem} \label{eqn-Ei*}
For $\bi\in\mathcal W^+$ with canonical factorization as in (\ref{eqn-i}) we have 
\begin{equation}
E^*_{\mathbf i}=(-1)^{\xi(\bi)}q^{s(\bi)}(E^*_{\mathbf i_d})^{\diamond n_d} \diamond \cdots \diamond (E^*_{\mathbf i_1})^{\diamond n_1}.
\end{equation}
\end{Lem}

\begin{proof} It follows from (\ref{eqn-Ei}) and (\ref{eqn-(Ei,Ei)}) that
\begin{align*}
E^*_{\mathbf i}&=\dfrac{E_{\bi}}{(E_{\bi},E_{\bi})}=(-1)^{\xi(\bi)}q^{s(\bi)}\prod_{k=1}^{d}\dfrac{[n_k]_{\bi_k}!}{(E_{\bi_k},E_{\bi_k})^{n_k}}E_{\bi}\\
&=(-1)^{\xi(\bi)}q^{s(\bi)}\left(\prod_{k=1}^{d}\dfrac{[n_k]_{\bi_k}!}{(E_{\bi_k},E_{\bi_k})^{n_k}}\right)E_{\mathbf i_d}^{(n_d)} \diamond \cdots \diamond E_{\mathbf i_1}^{(n_1)}\\
&=(-1)^{\xi(\bi)}q^{s(\bi)}\left(\prod_{k=1}^{d}\dfrac{[n_k]_{\bi_k}!}{(E_{\bi_k},E_{\bi_k})^{n_k}}\right)\dfrac{E_{\mathbf i_d}^{\diamond n_d}}{[n_d]_{\bi_d}!} \diamond \cdots \diamond \dfrac{E_{\mathbf i_1}^{\diamond n_1}}{[n_1]_{\bi_1}!}\\
&=(-1)^{\xi(\bi)}q^{s(\bi)}\dfrac{E_{\mathbf i_d}^{\diamond n_d}}{(E_{\bi_d},E_{\bi_d})^{n_d}} \diamond \cdots \diamond \dfrac{E_{\mathbf i_1}^{\diamond n_1}}{(E_{\bi_1},E_{\bi_1})^{n_1}}\\
&=(-1)^{\xi(\bi)}q^{s(\bi)}(E^*_{\mathbf i_d})^{\diamond n_d} \diamond \cdots \diamond (E^*_{\mathbf i_1})^{\diamond n_1}.
\end{align*}
\end{proof}

Define $\mathcal U_{\mathcal A}$ and $\mathcal U^*_{\mathcal A}$ to be the $\mathcal A$-subalgebras of $\mathcal U$ generated by $\{E_{\mathbf i}\,|\, \mathbf i \in \mathcal W^+\}$ and $\{E^*_{\mathbf i}\,|\, \mathbf i \in \mathcal W^+\}$, respectively.
 Then we have
\[ \mathcal U^*_{\mathcal A}=\{ v \in \mathcal U \,|\, (u,v)\in \mathcal A \text{ for all } u \in \mathcal U_{\mathcal A}\} . \] 
For $\mathbf i=\bi_1^{n_1} \cdots \bi_d^{n_d}$ as in (\ref{eqn-i}), set $\varsigma_{\mathbf i}=[n_1]_{\bi_1}! \cdots [n_d]_{\bi_d}!$ and consider the free $\mathcal A$-module $\mathcal F_{\mathcal A} = \bigoplus_{\mathbf i \in \mathcal W} \mathcal A \, \varsigma_{\mathbf i} \mathbf i$. Then we also have \[ \mathcal U^*_{\mathcal A}= \mathcal F_{\mathcal A} \cap \mathcal U, \] and  $\mathcal U^*_{\mathcal A}$ is $Q^+$-graded with 
$(\mathcal U^*_{\mathcal A})_{\alpha}=\mathcal F_{\mathcal A}\cap \mathcal U_{\alpha}.$

\begin{Thm}\cite[Theorem 7.11]{CHW}\label{thm-canonical basis}
There exists a basis $\{ b^*_{\mathbf i} \,|\, \mathbf i \in \mathcal W^+ \}$ for $\mathcal U_{\mathcal A}^*$ characterized by the properties:
\begin{enumerate}
\item $b^*_{\mathbf i} -E^*_{\mathbf i}$ is a linear combination of vectors $E^*_{\mathbf j}$, $\mathbf j \prec \mathbf i$, with coefficients $q \mathbb Z[q];$

\item If we write $b^*_{\mathbf i}=\sum_{\mathbf j} c_{\mathbf j}\  \mathbf j$, $c_{\mathbf j} \in \mathcal A$, then we have $\overline{c_{\mathbf j}}=c_{\mathbf j}$. (Recall $\bar q = -q^{-1}$.)
\end{enumerate}
 Furthermore, we have $\max(b^*_{\mathbf i})= \mathbf i$ for all $\mathbf i \in \mathcal W^+$, and $b^*_{\mathbf i}=E^*_{\mathbf i}$ for $\mathbf i \in \mathcal L^+$.
\end{Thm}

The basis $B^*= \{ b^*_{\mathbf i} \,|\, \mathbf i \in \mathcal W^+ \}$ is called the {\em dual canonical basis} for $\mathcal U^*_{\mathcal A}$.
Let $\beta\in\Phi^+$ and $n \in \mathbb Z_{\geq0}.$ Clearly, $E^*_{{\boi^+(\beta)}^n}\in \mathcal({U}^*_{\mathcal A})_{n\beta}.$ Since $\{b^*_{\bi}\mid |\bi|=n\beta\}$ is a basis of  $(\mathcal{U}^*_{\mathcal A})_{n\beta}$ and  by Corollary \ref{cor-Lyndon word power}, $\boi^+(\beta)^n$ is the smallest dominant word in $\mathcal W^+_{n\beta},$ it follows from  Theorem \ref{thm-canonical basis} that

\begin{Cor}\label{cor-b*word power}
$E^*_{{\boi^+(\beta)}^n}=b^*_{{\boi^+(\beta)}^n}$.
\end{Cor}





\medskip

\section{Spin quiver Hecke algebras} \label{spin-quiver}

\subsection{Generators and relations}Let $\mathbb K$ be a field with $\mathrm{char}\  \mathbb K \neq 2$, and let $\widetilde \Gamma$ be a quiver with compatible automorphism $a:\widetilde \Gamma \rightarrow \widetilde \Gamma$. Denote the set of vertices of $\widetilde \Gamma$ by $\widetilde I$ and the set of edges by $\widetilde H$. We have maps $s: \widetilde H \rightarrow \widetilde I$ and $t:\widetilde H \rightarrow \widetilde I$ such that $s(a(h))=a(s(h))$ and $t(a(h))=a(t(h))$ for all $h \in \widetilde H$. 
 Set $I$ to be a set of representatives of the orbits of  $\widetilde I$ under $a$ and let $\Gamma = \widetilde \Gamma/a$ be the Dynkin diagram with nodes labeled by $I$, assuming $\Gamma$ has no loops. For each $i \in I$, let $\alpha_i \in \widetilde I /a$ be the corresponding orbit.  For $i\neq j$, we set
\[ (\alpha_i, \alpha_i)=2|\alpha_i| \qquad \text{ and } \qquad (\alpha_i, \alpha_j)=-|\{ (i',j')\in \widetilde H| i' \in \alpha_i , j' \in \alpha_j\}| .\]
Then we obtain a generalized Cartan matrix $A=(a_{ij})$ and a matrix $D=\mathrm{diag}(s_1, \ldots , s_n)$ by setting $s_i=|\alpha_i|$ and $a_{ij} = (\alpha_i, \alpha_j)/s_i$. Note that $DA$ is symmetric.

Now we assume that the matrix $A$ is the same as in \eqref{eqn-ma}, and put the same $\mathbb Z/2\mathbb Z$-grading on $I$, i.e. $I_{\bar 1}=\{n\}$. The orbit $\alpha_i$ is to be identified with the simple root $\alpha_i$ of the Kac-Moody superalgebra $\mathfrak g$ associated to $A$, and we keep all the notations in the previous section. 

Define $d_{ij}=| \{ h \in \widetilde H | s(h) \in \alpha_i \text{ and } t(h) \in \alpha_j \} /a |$ for $i \neq j$. For $i, j \in I$, set
\[ \mathbb K_{ij}\{u,v\}= \mathbb K \langle u, v\rangle/\langle uv-(-1)^{p(i)p(j)}vu \rangle,  \]  
and define $Q_{ii}(u,v)=0$ and  \[ Q_{ij}(u,v)=(-1)^{d_{ij}}(u^{2/s_i}-v^{2/s_j}) \in \mathbb K_{ij}\{ u,v\} \qquad \text{ for }i\neq j.\]
Assume $\nu =\sum_{i\in I}c_i \alpha_i \in Q^+$ with $\sum_{i \in I} c_i=d$. Set
\[I^\nu=\{ \mathbf i =(i_1, \dots, i_d) \in I^d \, | \, \alpha_{i_1}+\cdots + \alpha_{i_d}=\nu \} .\] The symmetric groups $S_d$ acts on $I^\nu$ by place permutations; in particular, the transposition $s_r$ acts as \[s_r \cdot (i_1, \dots, i_r, i_{r+1}, \dots, i_d) = (i_1, \dots, i_{r+1}, i_r, \dots, i_d). \]
 
The $\mathbb K$-algebra $\mathcal H^-(\nu)$ with the identity $1_\nu$ is defined by the generators $e(\mathbf i)$ ($\mathbf i \in I^\nu$), $y_r$ ($r=1, \dots, d$), $\tau_s$ ($s=1, \dots, d-1$) satisfying the following relations: 
\begin{equation}
e(\mathbf i) e(\mathbf j)=\delta_{\mathbf i \mathbf j}e(\mathbf i)\;\text{for all}\; \mathbf i, \mathbf j \in I^\nu;\quad
\sum_{\mathbf i\in I^\nu} e(\mathbf i)= 1_\nu; \label{eq_relations1}
\end{equation}
\begin{equation}
y_r e(\mathbf i)=e(\mathbf i) y_r;\label{eq_relations2}
\end{equation}
\begin{equation}\tau_re(\mathbf i)=e(s_r \cdot \mathbf i) \tau_r;\label{eq_relations3}
\end{equation} 
\begin{equation}y_ry_se(\mathbf i)=(-1)^{p(i_r)p(i_s)}y_sy_re(\mathbf i)\;\;{\rm for}\;\;r \neq s;  \label{eq_relations4}
\end{equation} 
\begin{equation}\tau_r y_s e(\mathbf i)=(-1)^{p(i_r)p(i_{r+1})p(i_s)}y_s \tau_r e(\mathbf i)\;\;{\rm  for}\;\; s \neq r, r+1;\label{eq_relations5}
\end{equation}
\begin{equation}\tau_r \tau_s e(\mathbf i)=(-1)^{p(i_r)p(i_{r+1})p(i_s)p(i_{s+1})}\tau_s \tau_r e(\mathbf i)\;\;{\rm for}\;\;|s-r|>1;\label{eq_relations6}
\end{equation} 
\begin{equation}\tau_r y_{r+1} e(\mathbf i)=\begin{cases} ((-1)^{p(i_r)p(i_{r+1})}y_r \tau_r +1) e(\mathbf i) & \text{if } i_r=i_{r+1}, \\ (-1)^{p(i_r)p(i_{r+1})}y_r \tau_r e(\mathbf i) & \text{if } i_r \neq i_{r+1}; 
\end{cases}\label{eq_relations7}
\end{equation} 
\begin{equation}y_{r+1}\tau_r  e(\mathbf i)=\begin{cases} ((-1)^{p(i_r)p(i_{r+1})} \tau_ry_r +1) e(\mathbf i) & \text{if } i_r=i_{r+1}, \\ (-1)^{p(i_r)p(i_{r+1})} \tau_r y_r e(\mathbf i) & \text{if } i_r \neq i_{r+1};
\end{cases}\label{eq_relations8}
\end{equation} 
\begin{equation}\tau_r^2 e(\mathbf i)=Q_{i_r, i_{r+1}}(y_r, y_{r+1})e(\mathbf i);\label{eq_relations9}
\end{equation} 
\begin{equation}\label{eq_relations10}
\begin{array}{ll}
(\tau_r \tau_{r+1} \tau_r -\tau_{r+1} \tau_r \tau_{r+1})e(\mathbf i)  \\ = \begin{cases}
\left ( \frac { Q_{i_r, i_{r+1}}(y_{r+2}, y_{r+1}) - Q_{i_r, i_{r+1}}(y_{r}, y_{r+1})} {y_{r+2}-y_r} \right) e(\mathbf i) & \text{if }i_r =i_{r+2} \neq n, \\ (-1)^{p(i_{r+1})} (y_{r+2}-y_r)\left ( \frac { Q_{i_r, i_{r+1}}(y_{r+2}, y_{r+1}) - Q_{i_r, i_{r+1}}(y_{r}, y_{r+1})} {y^2_{r+2}-y^2_r} \right) e(\mathbf i)
& \text{if } i_r =i_{r+2}=n, \\
0 & \text{otherwise.}
\end{cases}
\end{array}
\end{equation}

Now the {\em spin quiver Hecke algebra} is defined to be 
$\displaystyle\mathcal H^-=\bigoplus_{\nu \in Q^+} \mathcal H^-(\nu)$. We define a $\mathbb Z$-grading on $\mathcal H^-$ by $\deg e(\mathbf i) =0$, $\deg y_r e(\mathbf i) =(\alpha_{i_r}, \alpha_{i_r})$ and $\deg \tau_r e(\mathbf i) = - (\alpha_{i_r}, \alpha_{i_{r+1}})$, 
and a $\mathbb Z/2 \mathbb Z$-grading by
$p(e(\mathbf i))=\bar 0$, $p(y_r e(\mathbf i))=p(i_r)$ and $p(\tau_r e(\mathbf i))=p(i_r)p(i_{r+1})$.

\subsection{Module categories} Let $\rm{Mod}^-(\nu)$ be the abelian category of finitely generated $(\mathbb Z \times \mathbb Z/2\mathbb Z)$-graded left $\mathcal H^-(\nu)$-modules. We write $\mathrm{Hom}_\nu$ for $\mathrm{Hom}_{\mathcal H^-(\nu)}$. For any $M \in \rm{Mod}^-(\nu)$, define its {\em $q$-super dimension} by 
\[ \dim^-_q M=\sum_{k \in \mathbb Z} ( \dim M_{\bar 0}[k]-\dim M_{\bar 1}[k])q^k \in \mathbb Z((q)),\] 
and define the {\em graded character} by
\[ \mathrm{ch}^-_q M=\sum_{\mathbf i \in I^\nu} (\dim^-_q e(\mathbf i)M)\, \mathbf i .\] The parity shift functor $\Pi : \rm{Mod}^-(\nu) \rightarrow \rm{Mod}^-(\nu)$ is defined by $(\Pi M)_{\bar 0} =M_{\bar 1}$ and $(\Pi M)_{\bar 1} = M_{\bar 0}$.   We denote by $M\{m\}$ the same $\mathcal H^-(\nu)$-module $M$ with the $\mathbb Z$-grading shifted by $m \in \mathbb Z$, i.e.  $M\{m\}[k]=M[k-m]$ for $k \in \mathbb Z$.  Then the grading shift functor $q: \rm{Mod}^-(\nu) \rightarrow \rm{Mod}^-(\nu)$ is defined by $qM=M\{1\}$.

Set $\mathrm{Hom}^-_\nu(M,N)= \mathrm{Hom}_\nu (M,N) \oplus \mathrm{Hom}_\nu(M,\Pi N)$ with $\mathbb Z/2 \mathbb Z$-grading given by \[ \mathrm{Hom}^-_\nu(M,N)_{\bar 0}=\mathrm{Hom}_\nu(M,N) \quad \text{ and } \quad \mathrm{Hom}^-_\nu(M,N)_{\bar 1}=\mathrm{Hom}_\nu(M, \Pi N),\]
and define \[ \mathrm{HOM}^-_\nu(M,N)=\bigoplus_{m \in \mathbb Z} \mathrm{Hom}^-_\nu(M, \Pi^m N\{m\}). \]

Let $\mathcal A=\mathbb Z[q, q^{-1}]$ as before. The full subcategory of $\rm{Mod}^-(\nu)$ consisting of finite dimensional (resp. finitely generated projective) modules is denoted by $\rm{Rep}^-(\nu)$ (resp. $\rm Proj^-(\nu)$), and the corresponding Grothendieck group by $[\rm Rep^-(\nu)]$ (resp.  $[\rm Proj^-(\nu)]$). The functors $\Pi$ and $q$ define $\mathcal A$-module structures on both $\rm Rep^-(\nu)$ and $\rm Proj^-(\nu)$ via $q[M]=[qM]$ and $-[M]=[\Pi M]$. 

There is a unique $\mathbb K$-linear anti-automorphism $\psi: \mathcal H^-(\nu) \rightarrow \mathcal H^-(\nu)$ defined by $\psi(\mathbf i) =\mathbf i$, $\psi(y_r)=y_r$ and $\psi(\tau_s)=\tau_s$ for all $\mathbf i \in I^\nu$ and $1 \le r \le d$, $1\le s <d$. For a graded right $\mathcal H^-(\nu)$-module $M$, we define $M^\psi$ to be the left module with the action given by $x. m=m. \psi(x)$ for $m\in M$ and $x \in \mathcal H^-(\nu)$. Similarly, for a graded left $\mathcal H^-(\nu)$-module, we denote by the same notation $M^\psi$ the right module with the action twisted by $\psi$.
Define $P^{\#}=\mathrm{HOM}^-_\nu(P, \mathcal H^-(\nu))^\psi$ for $P \in \rm Proj^-(\nu)$, and the $\mathbb Z$-linear bar-involution on $[\rm Proj^-(\nu)]$ by $\overline q=-q^{-1}$ and $\overline{[P]}=[P^{\#}]$.  We define a bilinear form $(\cdot,\cdot): [{\rm Proj}^-(\nu)] \times [{\rm Proj}^-(\nu)] \rightarrow \mathbb Z((q))$ by
\[ ([P], [Q]) = \dim^-_q(P^\psi \otimes_{\mathcal H^-} Q) = \dim^-_q \mathrm{HOM}^-_\nu(P^{\#}, Q) .\]

For $M \in \rm Rep^-(\nu)$, we define its graded dual $M^\circledast=\mathrm{HOM}^-_{\mathbb K}(M, \mathbb K)^{\psi}$ 
with the $\mathcal H^-(\nu)$-action given by $(x.f)(m)=f(\psi(x).m)$ for $x \in \mathcal H^-(\nu)$, $f \in M^\circledast$ and $m\in M$, where we set $\mathbb K =\mathbb K_{\bar 0}$.  Then we obtain $M^\circledast \in \rm Rep^-(\nu)$. A bar-involution on $[\rm Rep^-(\nu)]$ is defined by $\overline q =-q^{-1}$ and $\overline{[M]}=[M^\circledast]$. 
Define  an $\mathcal A$-pairing
$\langle \cdot, \cdot \rangle: [ \rm Proj^-(\nu)] \times [\rm Rep^-(\nu)] \longrightarrow \mathcal A$ by \[ \langle [P], [M] \rangle = \dim^-_q \mathrm{HOM}^-_\nu(P^{\#}, M). \] 
For each irreducible representation $L \in \rm Rep^-(\nu)$, there exists a projective indecomposable cover $P_L \in \rm Proj^-(\nu)$, which is dual to $L$ with respect to the pairing. Every element of $\rm Proj^-(\nu)$ is a direct sum of indecomposable representations of the form $P_L\{ m\}$ for some irreducible $L$ and $m \in \mathbb Z$. Thus the pairing $\langle \cdot, \cdot \rangle$ is a perfect pairing. 

\subsection{Induction and restriction}Let $\mu, \nu \in Q^+$, and set $1_{\mu, \nu}=\displaystyle{\sum_{\mathbf i \in I^\mu,\, \mathbf j \in I^\nu} e (\mathbf i \mathbf j) }$. We have the natural embedding $\mathcal H^-(\mu) \otimes \mathcal H^-(\nu) \hookrightarrow \mathcal H^-(\mu+\nu)$. Define the functor $\rm Res^{\mu+\nu}_{\mu, \nu} : Mod^-(\mu+\nu) \rightarrow Mod^-(\mu) \otimes Mod^-(\nu)$ by $\mathrm{Res}^{\mu+\nu}_{\mu, \nu} M=1_{\mu,\nu} M$, and the functor $\rm Ind^{\mu+\nu}_{\mu, \nu} :  Mod^-(\mu) \otimes Mod^-(\nu) \rightarrow Mod^-(\mu+\nu)$ by \[\mathrm{Ind}^{\mu+\nu}_{\mu,\nu}(M\otimes N)=\mathcal H^-(\mu+\nu) 1_{\mu,\nu}  \bigotimes_{\mathcal H^-(\mu) \otimes \mathcal H^-(\nu)} (M \boxtimes N).\]
Then we obtain the functors \[ \rm Ind=\bigoplus_{\mu,\nu} Ind^{\mu+\nu}_{\mu,\nu}\quad \text{ and } \quad \rm Res=\bigoplus_{\substack{\lambda, \mu, \nu \\ \mu+\nu=\lambda}}Res^{\lambda}_{\mu,\nu} .\]
Set $\rm [Proj^-]=\bigoplus_{\nu \in Q^+} [Proj^-(\nu)]$ and $\rm [Rep^-]=\bigoplus_{\nu \in Q^+} [Rep^-(\nu)]$. Then $[\rm Ind]$ defines a multiplication on $[\rm Proj^-]$ to make it an $\mathcal A$-algebra. Similarly, $[\rm Rep^-]$ becomes  an $\mathcal A$-algebra with  $\rm[Ind]$. Furthermore,  $[\rm Res]$ defines a comultiplication on both $[\rm Proj^-]$  and $[\rm Rep^-]$ to make them $\mathcal A$-coalgebras. 

\subsection{Categorification of $\mathcal U_{\mathcal A}$ and $\mathcal U^*_{\mathcal A}$} 

\begin{Thm}[\cite{HW,KKO14}] \label{thm-iso}
There exists a $\mathbb Z \times \mathbb Z/2\mathbb Z$-graded $\mathcal A$-algebra isomorphism $\gamma: \mathcal U_{\mathcal A} \xrightarrow{\sim} [\mathrm{Proj}^-]$  commuting with  the bar-involutions on $\mathcal U_{\mathcal A}$ and $\rm [Proj^-]$. 
\end{Thm}

\begin{Cor} \label{cor-iso}
The induced map $\gamma^*: [\mathrm{Rep}^-] \xrightarrow{\sim} \mathcal U^*_{\mathcal A}$ is an  isomorphism of $\mathbb Z \times \mathbb Z/2\mathbb Z$-graded $\mathcal A$-algebras. 
\end{Cor}


For  $M \in \rm{Mod}^-(\mu)$ and $ N \in  \rm{Mod}^-(\nu)$, we define \[M \circ N := \mathrm{Ind}^{\mu+\nu}_{\mu,\nu}(M\boxtimes N)\]

\begin{Prop}\label{cor-char}
Let $\mu, \nu\in Q^+,$ $M\in \mathrm{Rep}^-(\mu)$ and $N\in \mathrm{Rep}^-(\nu).$ Then we have \[\mathrm{ch}^-_q(M\circ N)=\mathrm{ch}^-_q(N)\diamond \mathrm{ch}^-_q(M).\]
\end{Prop}

\begin{proof}
Choose a homogeneous basis $\{ v_1, \dots , v_k \}$ for $e(\bi)M$ and $\{ u_1, \dots , u_l \}$ for $e(\bj)N$. Then we obtain a basis $\{ \tau_w v_p \otimes u_q : 1 \le p \le k, \ 1 \le q \le l , w \in S_{a+b}/S_a \times S_b \}$ of the homogeneous space of $M \circ N$. 
One can see that we need only to prove \[ \sum_{w \in S_{a+b}/S_a \times S_b} c(\tau_w) w (\bi \bj) =   \bj \diamond \bi , \] where $c(\tau_w)=(-1)^{p(\tau_we(\bi \bj))} q ^{\deg (\tau_w e(\bi \bj))}$.

We use the inductive formula \eqref{ind}. Consider $\bi=(i_1, \dots , i_a)$ and $\bj=(i_{a+1}, \dots , i_{a+b})$. Then we have 
\begin{align*} & \sum_{w \in S_{a+b}/S_a \times S_b} c(\tau_w) w (\bi \bj) \\ &= \left ( \sum_{w \in S_{(a-1)+b}/S_{(a-1)} \times S_b} c (\tau_w) \right ) c(\tau_{a+b-1} \cdots \tau_{a+1} \tau_a) w \tau_{a+b-1} \cdots \tau_{a+1} \tau_a (\bi \bj) \\ &  \qquad + \sum_{w \in S_{a+(b-1)}/S_a \times S_{(b-1)}} c(\tau_w) w(\bi (i_{a+1}, \dots , i_{a+b-1}))i_{a+b}.\end{align*}

On the other hand, we have
\begin{align*} & \bj \diamond \bi = (i_{a+1}, \dots , i_{a+b}) \diamond (i_1, \dots , i_a) \\ & = (-1)^{p(xi_{a+b})p(i_a)}q^{-(|xi_{a+b}|, |i_a|)} ((i_{a+1}, \dots , i_{a+b}) \diamond (i_1, \dots , i_{a-1})   ) i_a \\ & \qquad + ((i_{a+1}, \dots , i_{a+b-1})\diamond (i_1,\dots , i_a)    )i_{a+b},
\end{align*}
where $x=(i_{a+1}, \dots , i_{a+b-1})$. 
Since $c (\tau_{a+b-1} \cdots \tau_{a+1} \tau_a) =  (-1)^{p(xi_{a+b})p(i_a)}q^{-(|xi_{a+b}|, |i_a|)}$, we are done by induction.
\end{proof}

We have the following important property of the map $\mathrm{ch}^-_q$, which is proved in \cite[Corollary 8.16]{KKO14}.

\begin{Prop}[\cite{KKO14}] \label{injective}
Let $M, M' \in \mathrm{Rep}^-(\mu)$. If $\mathrm{ch}^-_q(M)=\mathrm{ch}^-_q(M')$, then $[M]=[M']$.
\end{Prop}

\medskip

\section{Cuspidal representations} \label{cuspidal}

In this section we give an explicit construction of the cuspidal modules with the ordering we fixed on $I$: $1 \prec 2 \prec \dots \prec n$. These cuspidal modules will be building blocks for irreducible modules. We begin with the definition of a cuspidal module.

\begin{Def} \hfill
\begin{enumerate}
\item Let $\nu \in Q^+$ and $M \in \mathrm {Rep}^-(\nu)$. The word $\mathbf i = \max (\mathrm{ch}^-_q M)$ is called the {\em highest weight} of $M$.

\item Let $\alpha\in \Phi^{+}$. An irreducible $\mathcal H^{-}(\alpha)$-module $L$ is called {\em cuspidal} if the highest weight of $L$ is a dominant Lyndon word. 
\end{enumerate}
\end{Def}



The set $\Phi^+$ of reduced positive roots is $$\Phi^+=\{\alpha(i,j)\mid1\le i\le j\leq n\}\cup\{\beta(i,j)\mid 1\leq i < j\leq n\},$$ where $\alpha(i,j):=\sum_{r=i}^{j}\alpha_r$ and $\beta(i,j):=\sum_{r=i}^{j-1}\alpha_r+\sum_{r=j}^{n}2\alpha_r$. 
It follows from Proposition \ref{pro-lyn} that the corresponding dominant Lyndon words are:
\begin{align*}
\boi^{+}(\alpha(i,j))=(i,\dots, j),\quad 1\le i\le j\leq n, \quad
\boi^{+}(\beta(i,j))=(i,\dots, n,n,\dots j), \quad  1\leq i < j\leq n.
\end{align*}
The corresponding dual canonical bases elements are 
\begin{align*}
E^*_{\boi^+(\alpha(i,j))}&=(i,\dots, j) , &  1\le i\le j\leq n, \\
E^*_{\boi^+(\beta(i,j))}&=(-q+q^{-1})(i,\dots,n,n,\dots,j),&   1\leq i < j\leq n.
\end{align*}

\begin{Prop} \label{prop-cusp} 
Let $\alpha\in \Phi^+$. 
For $\alpha=\alpha(i,j)$, we have the corresponding $1$-dimensional cuspidal module $L_{\alpha}=\mathbb{K}v_{\alpha}$ with the action of the generators: 
$$
e(\bj)v_{\alpha}=\delta_{\bj,\boi^{+}(\alpha)}v_{\alpha},\quad \tau_rv_{\alpha}=y_rv_{\alpha}=0\;\; \text{\rm for all}\;\; r.
$$
For $\alpha=\beta(i,j)$, $i<j$, we have the $2$-dimensional cuspidal module 
$L_{\alpha}:=\mathbb{K} v_1\oplus\mathbb{K}v_{-1}$, where $\deg v_g=g$ for $g=\pm 1$, and $p(v_1)=\bar 1$, $p(v_{-1})=\bar 0$, and the action of generators are given by: 
\[ \begin{array}{lll}
& e(\bj)v_g=\delta_{\bj,\boi^{+}(\alpha)}v_g  & \quad \text{ for }  g=\pm 1;
\\
& y_rv_1=0   & \quad \text{ for all }  r ;
\\
& y_rv_{-1}=0  & \quad \text{ if } r\neq n-i+1, n-i+2;
\\
& y_rv_{-1}=v_1  & \quad \text{ if }  r= n-i+1, n-i+2;
\\
& \tau_r v_1=0  & \quad \text{ if } r\neq n-i+1;
\\
& \tau_{n-i+1}v_1=v_{-1}; & 
\\
& \tau_rv_{-1}=0  & \quad \text{ for all }  r.
\end{array} \]

\end{Prop}
\begin{proof}
If $\alpha=\alpha(i,j)$, then it is straightforward to check that the action satisfies (\ref{eq_relations1})--(\ref{eq_relations10}). We clearly have  $\mathrm{ch}^-_q (L_\alpha)=\boi^+(\alpha)=(i,\dots, j)$. Thus $L_\alpha$ is a cuspidal representation.

Assume that $\alpha=\beta(i,j)$, $i<j$, and consider the action of generators on $L_\alpha=\mathbb{K} v_1\oplus\mathbb{K}v_{-1}$. 
Clearly the relation (\ref{eq_relations1}) holds. For $g=\pm 1$ , notice that 
\[ \begin{array}{ll} y_re(\bi)v_g=y_rv_g=e(\bi)y_r v_g , & \text{ if } \bi =\boi^+(\alpha); \\ y_r e(\bi)v_g=0=e(\bi)y_r v_g & \text{ otherwise}. \end{array} \] Thus the relation (\ref{eq_relations2}) holds.  

We have \[ \tau_r e(\bi) v_g = 
\begin{cases} v_{-1} & \text{ if } \bi=\boi^+(\alpha), \ r=n-i+1, \ g=1;  \\ 0 & \text{ otherwise}. \end{cases} \]
Since $e(s_r \cdot \bi)=e(\bi)$ for $r=n-i+1$, we get 
\begin{align*} e(s_r \cdot \bi)  \tau_r v_g & = 
\begin{cases} e(s_r \cdot \bi) v_{-1} & \text{ if }  r=n-i+1, \ g=1,  \\ 0 & \text{ otherwise}, \end{cases} \\ & = 
\begin{cases} v_{-1} & \text{ if } \bi=\boi^+(\alpha), \ r=n-i+1, \ g=1,  \\ 0 & \text{ otherwise}. \end{cases}\end{align*}
Thus $\tau_r e(\bi) v_g=e(s_r \cdot \bi)  \tau_r v_g $ for $g =\pm 1$, and the relation \eqref{eq_relations3} holds.

For the relations \eqref{eq_relations4}--\eqref{eq_relations10}, we may assume that $\bi=\boi^+(\alpha)$ and will drop $e(\bi)$ from consideration. Since $y_ry_s v_g=0$ for any $r,s$ and $g=\pm 1$, the relation \eqref{eq_relations4} is valid. For the relation \eqref{eq_relations5}, we assume that $s\neq r, r+1$. Then
\[ \tau_r y_s v_g = \begin{cases} \tau_r v_1 & \text{ if }  s=n-i+1, n-i+2,  \ g=-1;  \\ 0 & \text{ otherwise}. \end{cases} \]
Since $r \neq n-i+1$ from the assumption, we obtain $\tau_r y_s v_g=0$. Similarly, $y_s \tau_r v_g=0$, and the relation \eqref{eq_relations5} holds. Next we have  $\tau_r\tau_s v_g=0$ for any $r,s$, and the relation \eqref{eq_relations6} is valid.

Now we see
\begin{align*}   \tau_r y_{r+1} v_g & = 
\begin{cases} \tau_r v_1 & \text{ if }  r=n-i, n-i+1, \ g=-1,  \\ 0 & \text{ otherwise}, \end{cases} \\ & = 
\begin{cases} v_{-1} & \text{ if }  \ r=n-i+1, \ g=-1,  \\ 0 & \text{ otherwise}. \end{cases}\end{align*}
On the other hand, if $r=n-i+1, g=-1$, then $i_r=i_{r+1}=n$ and
\[ ((-1)^{p(i_r)p(i_{r+1})}y_r\tau_r+1)v_g=(-y_r \tau_r +1)v_{-1}=v_{-1}. \]
If $r=n-i+1$, $g=1$, then 
\[ ((-1)^{p(i_r)p(i_{r+1})}y_r\tau_r+1)v_g=(-y_r \tau_r +1)v_1=-y_rv_{-1}+v_1=-v_1+v_1=0 . \]
If $r\neq n-i+1$ then $i_r \neq i_{r+1}$ and $(-1)^{p(i_r)p(i_{r+1})}y_r\tau_r v_g=0$. 
Consequently, the relation \eqref{eq_relations7} holds.
The relation \eqref{eq_relations8} can be verified similarly, and we omit the details.

Clearly, $\tau_r^2 v_g=0$ for any $r$ and $g=\pm 1$. On the other hand, if $r\neq n-i, n-i+2$ then we obtain immediately $Q_{i_r, i_{r+1}}(y_r, y_{r+1})v_g=0$ . If $r=n-i$ then \[ Q_{i_r, i_{r+1}}(y_r, y_{r+1})v_g=Q_{n-1,n}(y_r, y_{r+1})v_g= \pm (y_r -y_{r+1}^2)v_g=0. \]
Similarly, if $r=n-i+2$ then 
\[ Q_{i_r, i_{r+1}}(y_r, y_{r+1})v_g=Q_{n,n-1}(y_r, y_{r+1})v_g= \pm (y_r^2 -y_{r+1})v_g=0. \]
Thus we see that the relation \eqref{eq_relations9} holds.
Finally, $(\tau_r \tau_{r+1} \tau_r - \tau_{r+1} \tau_r \tau_{r+1}) v_g =0$ for any $r$ and $g=\pm 1$, while $i_r\neq i_{r+2}$ for any $r$. Hence it is easy to see that the relation \eqref{eq_relations10} is valid.

Now we have shown that all the relations \eqref{eq_relations1}--\eqref{eq_relations10} are compatible with the action of the generators on the module $L_\alpha$, making it indeed an $\mathcal H^-(\alpha)$-module. Furthermore, \[ \mathrm{ch}^-_q (L_\alpha)=(-q+q^{-1})\boi^+(\alpha)=(-q+q^{-1})(i,\dots,n,n \dots,  j) . \] Thus $L_\alpha$ is a cuspidal representation for $\alpha=\beta(i,j)$, $i<j$.
\end{proof}

\begin{Cor} \label{cor-alpha}
We have $\mathrm{ch}^-_q(L_{\alpha})=E^*_{\boi^+(\alpha)}$ for $\alpha \in \Phi^+$.
\end{Cor}
\medskip

\section{Standard representations} \label{standard}

In this section, we use the results of the previous sections and construct all the irreducible representations of the spin quiver Hecke algebra to obtain the main result of this paper. 

\medskip

Recall that we have the dual canonical basis $B^*=\{ b^*_{\mathbf i} \,|\, \mathbf i \in \mathcal W^+ \}$  for $\mathcal U^*_{\mathcal A}$. Denote the coefficient of $\mathbf i$ in $b^*_{\mathbf i}$ by $\kappa_{\mathbf i}$. 
\begin{Lem} \label{lem-coeff}
 Assume that the canonical factorization of $\mathbf i = \mathbf i_1^{n_1} \cdots \mathbf i_d^{n_d} \in \mathcal W_\alpha^+$ is as in \eqref{eqn-i}. Then $\displaystyle\kappa_{\mathbf i} = \prod_{k=1}^d \kappa_{\mathbf i_k}^{n_k} [ n_k]_{\bi_k}!$. 
\end{Lem}

\begin{proof}
As in the proofs of \cite[Proposition 39, Theorem 40]{Lec} (see also \cite[Theorem 7.11]{CHW}), we have that $\displaystyle{b^*_{\bi}=E^*_{\bi}+\sum_{\bj\prec\bi}\gamma_{\bi\bj}E^*_{\bj}.}$ Since $\mathrm{max}(E^*_{\bi})=\mathrm{max}(E_\bi)=\bi$, it suffices to compute the coefficient of $\bi$ in $E^*_{\bi}=(-1)^{\xi(\bi)}q^{s(\bi)}(E^*_{\mathbf i_d})^{\diamond n_d} \diamond \cdots \diamond (E^*_{\mathbf i_1})^{\diamond n_1}.$  Since $\bi_k\in\mathcal L^+,$ we have that $E^*_{\bi_k}=b^*_{\bi_k}$ for all $k=1,\cdots, d,$ and the coefficient of $\bi_k$ in $E_{\bi_k}^*$ is $\kappa_{\bi_k}.$ Now, by Corollary \ref{cor-top coefficient}, we have that the coefficient of $\bi$ in $\mathbf i_d^{\diamond n_d}\diamond \cdots\diamond \mathbf i_1^{\diamond n_1}$ is $(-1)^{\xi(\mathbf i)}q^{-s(\mathbf i)}\displaystyle\prod_{k=1}^d[n_{k}]_{\bi_k}!.$ Hence the coefficient of $\bi$ in $E_{\bi}^*$ is $\displaystyle\prod_{k=1}^d \kappa_{\mathbf i_k}^{n_k} [ n_k]_{\bi_k}!.$
\end{proof}

\begin{Lem} \label{lem-power}
Let $\beta \in \Phi^+$ and $n \in \mathbb Z_{>0}$. Then $L_\beta^{\circ n}$ is irreducible with highest weight $\boi^+(\beta)^n$, and $\mathrm{ch}^-_q (L_\beta^{\circ n})=(-1)^{\xi(\boi^+(\beta)^n)}q^{-s(\boi^+(\beta)^n)} b^*_{{\boi^+(\beta)}^n}$.
\end{Lem}

\begin{proof} 
Recall that since $\boi^+(\beta)\in \mathcal L^+$ we have $b^*_{\boi^+(\beta)}=E^*_{\boi^+(\beta)}.$ By Corollary \ref{cor-alpha} we have $\mathrm{ch}^-_q(L_{\beta})=E^*_{\boi^+(\beta)}.$ From Proposition \ref{cor-char}, Lemma \ref{eqn-Ei*} and Corollary \ref{cor-b*word power}, it follows that 
\begin{align*}
\mathrm{ch}^-_q(L_{\beta}^{\circ n})&=(\mathrm{ch}^-_q(L_{\beta}))^{\diamond n} \\
&=({E^*_{\boi^+(\beta)}})^{\diamond n}\\
&=(-1)^{\xi(\boi^+(\beta)^n)}q^{-s(\boi^+(\beta)^n)}E^*_{{\boi^+(\beta)}^n}=(-1)^{\xi(\boi^+(\beta)^n)}q^{-s(\boi^+(\beta)^n)} b^*_{{\boi^+(\beta)}^n}.
\end{align*}
Hence all composition factors of $L_\beta^{\circ n}$ have
highest weight $\boi^+(\beta)^n$. Recall that the map $\mathrm{ch}^-_q$ is injective by Proposition \ref{injective}. Since $\{b^*_{\bi}\mid |\bi|=n\beta\}$ is a basis of  $(\mathcal{U}^*_{\mathcal A})_{n\beta}$ and $\boi^+(\beta)^n$ is the smallest dominant word in $\mathcal W^+_{n\beta},$ the representation $L_\beta^{\circ n}$ is irreducible.
\end{proof}

Consider $\mathbf i \in \mathcal W_\alpha^+$ and write it in the form of the canonical factorization $\mathbf i = \mathbf i_1^{n_1} \cdots \mathbf i_d^{n_d}$. Let $\beta_k =| \mathbf i_k|$ for $k=1, ... , d$, and define the {\em standard module} $\Delta (\mathbf i)$ of highest weight $\mathbf i  \in \mathcal W_\alpha^+$ over the algebra $\mathcal H^-(\alpha)$ by
\[  \Delta (\mathbf i):= \Pi^{\xi(\bi)}(L_{\beta_1}^{\circ n_1} \circ L_{\beta_2}^{\circ n_2} \circ \cdots \circ L_{\beta_d}^{\circ n_d} )  \{ s(\mathbf i) \}  \] where $\xi(\bi)$ and $s(\bi)$ are defined in (\ref{eq-xi-s}).

\begin{Lem} \label{lem-std}
Let $\mathbf i \in \mathcal W^+$. Then the highest weight of $\Delta(\mathbf i)$ is $\mathbf i$, and $\dim^-_q ( \mathbf i \Delta(\mathbf i))=\kappa_{\mathbf i}$. 
\end{Lem}

\begin{proof}
It is easy to see that  $ \xi(\mathbf i)= \sum_{k=1}^d \xi({\boi^+(\beta_k)}^{n_k})$ and $s(\mathbf i)= \sum_{k=1}^d s({\boi^+(\beta_k)}^{n_k})$. It follows  from  Proposition \ref{cor-char} and Lemma \ref{eqn-Ei*} that 
\begin{align*}
\mathrm{ch}^-_q(\Delta(\mathbf i))&=(-1)^{\xi(\mathbf i)}q^{s(\mathbf i)}\mathrm{ch}^-_q(L_{\beta_d})^{\diamond n_d}\diamond \cdots \diamond \mathrm{ch}^-_q(L_{\beta_1})^{\diamond n_1}\\
&=(-1)^{\xi(\mathbf i)}q^{s(\mathbf i)}(E^*_{\boi^+(\beta_d)})^{\diamond n_d}\diamond \cdots \diamond (E^*_{\boi^+(\beta_1)})^{\diamond n_1}\\
&=E^*_{\mathbf i}
\end{align*}
Hence the highest weight of $\Delta(\bi)$ is $\max(E^*_{\bi})=\bi$  and $\dim^-_q ( \mathbf i \Delta(\mathbf i))=\kappa_{\mathbf i}$ by  Lemma \ref{lem-coeff}.
\end{proof}

For $\mu, \nu \in Q^+$,  we will write $\mathrm{Hom}_{\mu+\nu}$ for $\mathrm{Hom}_{\mathcal{H}^-(\mu) \otimes \mathcal{H}^-(\nu)}$ and recall that we write $\mathrm{Hom}_\nu$ for $\mathrm{Hom}_{\mathcal H^-(\nu)}$. For $M\in \mathrm{Mod}^-(\mu) \otimes \mathrm{Mod}^-(\nu)$  and  $N \in \mathrm{Mod}^-(\mu+\nu)$, we have the Frobenius reciprocity:
\[ \mathrm{Hom}_{\mu+\nu}( \mathrm{Ind}^{\mu+\nu}_{\mu, \nu} M, N) \cong \mathrm{Hom}_{\mu,\nu}(M, \mathrm{Res}^{\mu+\nu}_{\mu,\nu}N). \]

\begin{Prop} \label{construction}
Let $\mathbf i \in \mathcal W^+_\alpha$, $\alpha \in Q^+$. Then the standard module $\Delta (\mathbf i)$ has an irreducible head, which will be denoted by $L(\mathbf i)$, and the highest weight of $L(\mathbf i)$ is $\mathbf i$. 
\end{Prop}

\begin{proof}

Let $L \in \text{Rep}(\alpha)$ be irreducible. If $L$ is a component of the head of $\Delta (\mathbf i)$, then $\mathrm{Hom}_\alpha(\Delta(\mathbf i), L)$ is nonzero and equal to 
\[ \mathrm{Hom}_{n_1 \beta_1, \dots, n_d\beta_d}(\Pi^{\xi(\bi)}(L_{\beta_1}^{\circ n_1} \boxtimes L_{\beta_2}^{\circ n_2} \boxtimes \cdots \boxtimes L_{\beta_d}^{\circ n_d} )  \{ s(\mathbf i) \}, \mathrm{Res}^\alpha_{n_1\beta_1, \dots , n_d\beta_d}L) \]
by the Frobenius reciprocity.
By Lemma \ref{lem-power}, the $\mathcal H^-(n_1 \beta_1) \otimes \cdots \otimes \mathcal H^-(n_d \beta_d)$-module $\Pi^{\xi(\bi)}(L_{\beta_1}^{\circ n_1} \boxtimes L_{\beta_2}^{\circ n_2} \boxtimes \cdots \boxtimes L_{\beta_d}^{\circ n_d} )  \{ s(\mathbf i) \}$ is irreducible and embeds into $L$. It follows from Lemma \ref{lem-std} that the multiplicity of the weight $\mathbf i$ in  $\Pi^{\xi(\bi)}(L_{\beta_1}^{\circ n_1} \boxtimes L_{\beta_2}^{\circ n_2} \boxtimes \cdots \boxtimes L_{\beta_d}^{\circ n_d} )  \{ s(\mathbf i) \}$ is equal to that of the weight $\mathbf i$ in $\Delta(\mathbf i)$. 
Thus the head of $\Delta (\mathbf i)$ is irreducible.
\end{proof}

Now we state and prove the main result of this paper.

\begin{Thm} \label{thm-main}
Let $\alpha \in Q^+$. Then the set $\{L(\mathbf i) \, | \, \mathbf i \in \mathcal W^+_\alpha \}$ is a complete and irredundant set of irreducible graded $\mathcal H^-(\alpha)$-modules up to isomorphism and degree shift.
\end{Thm}

\begin{proof}
By Proposition \ref{construction}, we have constructed an irreducible module $L(\mathbf i)$ for each $\mathbf i \in \mathcal W^+_\alpha$. Furthermore, since the highest weights are different, we have $L(\mathbf i) \not\cong L(\mathbf j)$ for $\mathbf i \neq \mathbf j$. We have the basis $B^*= \{ b^*_{\mathbf i} \,|\, \mathbf i \in \mathcal W^+ \}$ for $\mathcal U^*_{\mathcal A}$, and a basis of the weight space $(\mathcal U^*_{\mathcal A})_\alpha$ is given by $\{ b^*_{\mathbf i} \in B^* \,|\, \mathbf i \in \mathcal W^+_\alpha \}$. Now the assertion of the theorem follows from Corollary \ref{cor-iso}.
\end{proof}

%
%

\bibliographystyle{amsplain}
\bibliography{KLR}

\end{document}